\documentclass{amsart}
\usepackage{booktabs, colonequals, comment, enumitem, tikz, xcolor}
\usepackage[hidelinks]{hyperref}

\usetikzlibrary{arrows.meta}

\date\today

\newtheorem{corollary}{Corollary}
\newtheorem{lemma}{Lemma}
\newtheorem{proposition}{Proposition}
\newtheorem{theorem}{Theorem}

\theoremstyle{definition}

\newtheorem{example}{Example}

\theoremstyle{remark}
\newtheorem{remark}{Remark}

\DeclareMathOperator{\Cay}{Cay}
\DeclareMathOperator{\character}{char}
\DeclareMathOperator{\id}{id}

\DeclareMathOperator{\Sk}{Sk}

\begin{document}
\title{Involutions in the Cayley--Dickson Construction}
\author{Masood Aryapoor}
\address[Masood Aryapoor]{Department of Business and Mathematics, M\"alar\-dalen University, SE-631 05 Eskilstuna, Sweden}
\email{masood.aryapoor@mdu.se}

\author{Per B\"ack}
\address[Per Bäck]{Department of Business and Mathematics, M\"alar\-dalen University,  SE-721 23  V\"aster\r{a}s, Sweden}
\email[Corresponding author]{per.back@mdu.se}

\author{Sophie Pautrel}
\address[Sophie Pautrel]{Department of Business and Mathematics, M\"alar\-dalen University,  SE-721 23  V\"aster\r{a}s, Sweden}
\email{sophie.pautrel@tuni.fi}


\begin{abstract}We determine all involutions in the Cayley--Dickson construction that extend the involution of the original \(*\)-algebra. We also find all algebra isomorphisms between the resulting Cayley doubles that extend the identity automorphism of the original \(*\)-algebra, and consequently classify the resulting \(*\)-algebras up to such \(*\)-algebra isomorphisms. As applications, we show that Cayley doubles of algebras without zero divisors admit exactly one additional extension of the original involution in the \(2\)-torsion-free case and none otherwise, prove that the classical Cayley--Dickson involution is the unique scalar involution, and obtain a classification of the \(*\)-algebras arising from \(\mathbb{R}\) up to dimension \(4\).
\end{abstract}

\maketitle

\section{Introduction}
To paraphrase McCrimmon~\cite[p.~64]{McC04}, the Cayley--Dickson construction, as its name suggests, is due to Albert~\cite{Alb42}; introduced by Dickson~\cite{Dic19} from earlier work of Cayley~\cite{Cay45}, it was later generalized by Albert~\cite{Alb42} into its present form.

The construction produces new \(*\)-algebras, that is, algebras equipped with an involution, from old ones by a doubling process. Starting from the real numbers \(\mathbb{R}\) and iterating the process gives, in particular, the complex numbers \(\mathbb{C}\), the quaternions \(\mathbb{H}\), and the octonions \(\mathbb{O}\); together with \(\mathbb{R}\), these are precisely the real normed division algebras. 

Despite its importance, the construction retains a somewhat ad hoc character. In~\cite{AB25}, the first and second authors showed that all Cayley doubles can instead be realized as quotients of noncommutative nonassociative polynomial algebras, extending the classical construction of \(\mathbb{C}\) and \(\mathbb{H}\) as polynomial quotients. 

In this article, we investigate another aspect of the Cayley--Dickson construction: the choice of involution. Given a Cayley double, we determine all involutions extending the involution of the original \(*\)-algebra (\autoref{thm:inv}). We then find all algebra isomorphisms between Cayley doubles that extend the identity automorphism of the original \(*\)-algebra (\autoref{thm:autom}), and consequently classify the resulting \(*\)-algebras up to such \(*\)-algebra isomorphisms (\autoref{prop:starisom}). As applications, we show that Cayley doubles of algebras without zero divisors admit exactly one additional extension of the original involution in the \(2\)-torsion-free case and none otherwise (\autoref{cor:alpha-beta-involutions} and \autoref{rem:alpha-beta-equality}), prove that the involution appearing in the classical construction is the unique scalar involution (\autoref{prop:nuclear-inv}), and obtain a classification of the \(*\)-algebras arising from \(\mathbb{R}\) up to dimension \(4\) (\autoref{prop:cay-over-R}).

The article is organized as follows. In \autoref{sec:preliminaries}, we provide preliminaries on \(*\)-algebras. In \autoref{sec:involutions}, we determine all involutions of Cayley doubles extending a given involution, and in \autoref{sec:isomorphisms} we determine the corresponding algebra isomorphisms and \(*\)-algebra isomorphisms.

\section{Preliminaries}\label{sec:preliminaries}
All rings and algebras are assumed to be unital. By an \emph{algebra}, we mean an algebra \(A\) over a commutative associative ring \(K\), which is not necessarily commutative or associative itself.

An \emph{involution of \(A\)} is a \(K\)-linear map
\[
*\colon A\to A,\qquad x\mapsto x^*,
\]
satisfying
\[
(x^*)^*=x
\quad\text{and}\quad
(xy)^*=y^*x^*
\qquad\text{for all }x,y\in A.
\]
An algebra equipped with an involution is called a \emph{\(*\)-algebra}.

Let \(A\) be a \(*\)-algebra and let \(\mu\in K\setminus\{0\}\). The \emph{Cayley double of \(A\)}, denoted by \(\Cay(A,\mu)\), is the algebra over \(K\) whose underlying \(K\)-module is \(A\oplus A\) and whose multiplication is given by
\[
(u,v)(x,y)\colonequals (ux+\mu y^*v, vx^*+yu)
\qquad\text{for all }u,v,x,y\in A.
\]
An involution \(\gamma\) of \(\Cay(A,\mu)\) is called an \emph{\(A\)-involution} if it extends the involution \(*\) of \(A\), that is,
\[
\gamma(x,0)=(x^*,0)
\qquad\text{for all }x\in A.
\]
We write \(\Cay_\gamma(A,\mu)\) for the \(*\)-algebra obtained by equipping \(\Cay(A,\mu)\) with the \(A\)-involution \(\gamma\). The standard \(A\)-involution \(\alpha\) of \(\Cay(A,\mu)\) is defined by
\[
\alpha(x,y)\colonequals (x^*,-y)
\qquad\text{for all }x,y\in A.
\]

We also use the following standard notions from nonassociative algebra. The \emph{commutator} \([\cdot,\cdot]\colon A\times A\to A\) and the \emph{associator} \((\cdot,\cdot,\cdot)\colon A\times A\times A\to A\) are defined, respectively, by
\[
[x,y]\colonequals xy-yx
\quad\text{and}\quad
(x,y,z)\colonequals (xy)z-x(yz)
\qquad\text{for all }x,y,z\in A.
\]
We denote by \([A,A]\) the \(K\)-submodule of \(A\) generated by all commutators \([x,y]\).

The \emph{left, middle,} and \emph{right nuclei of \(A\)} are

\begin{align*}
N_l(A)
&\colonequals
\{x\in A \mid (x,y,z)=0 \text{ for all } y,z\in A\},\\
N_m(A)
&\colonequals
\{y\in A \mid (x,y,z)=0 \text{ for all } x,z\in A\},\\
N_r(A)
&\colonequals
\{z\in A \mid (x,y,z)=0 \text{ for all } x,y\in A\}.
\end{align*}
Their intersection
\[
N(A)\colonequals N_l(A)\cap N_m(A)\cap N_r(A)
\]
is called the \emph{nucleus of \(A\)}. The left, middle, and right nuclei are associative subalgebras of \(A\); hence, so is the nucleus. 

The \emph{commuter of \(A\)} is
\[
C(A)\colonequals\{x\in A\mid [x,y]=0 \text{ for all } y\in A\},
\]
and the \emph{center of \(A\)} is
\[
Z(A)\colonequals N(A)\cap C(A).
\]
In particular, \(Z(A)\) is a commutative associative subalgebra of \(A\). We also define
\[
C_*(A)\colonequals\{x\in C(A)\mid x^*=x\}
\quad\text{and}\quad
Z_*(A)\colonequals\{x\in Z(A)\mid x^*=x\}.
\]

The algebra \(A\) is called \emph{\(2\)-torsion-free} if \(2x=0\) implies \(x=0\) for all \(x\in A\). Finally, we say that \(\mu\in K\setminus\{0\}\) is \emph{cancellable} if \(\mu x=0\) implies \(x=0\) for all \(x\in A\).

\section{Involutions}\label{sec:involutions}
\subsection{Cayley elements, skew elements, and \texorpdfstring{\(*\)}{*}-elements}
Let us recall some definitions from \cite{McC85} for elements of a \(*\)-algebra \(A\). 

An element \(a\in A\) is called \emph{skew} if \[a^*=-a.\]

We say that \(a\in A\) is a \emph{\(*\)-element} if
\[
ax=x^*a\qquad\text{for all }x\in A.
\]

Finally, an element \(a\in A\) is called \emph{Cayley} if
\[
(xy)a = y(xa) = (xa)y^*\qquad\text{for all }x,y\in A.
\]

In the following lemma, we collect several properties of these elements.

\begin{lemma}\label{lem:1}
Let \(A\) be a \(*\)-algebra and let \(a\in A\). Then the following hold:
\begin{enumerate}[label={(\roman*)}]
\item \(a\) is a skew \(*\)-element if and only if\[
(ax)^*=-ax \qquad\text{for all }x\in A\]
if and only if 
\[
(xa)^*=-xa \qquad\text{for all }x\in A.
\]
\label{lem:1i}
    \item If \(a\) is Cayley, then all \(ax\)   are \(*\)-elements, hence \(a\) is a \(*\)-element.  \label{lem:1ii}
    \item \(a\) is Cayley if and only if
    \[      a(xy)=(ay)x=x^*(ay)
        \qquad \text{for all } x,y\in A.
    \] \label{lem:1iii}
    \item If all \(ax\) are \(*\)-elements, then
    \[       a\in N_l(A) \iff a\in N_r(A).
    \] \label{lem:1iv}
    \item If \(a\) is Cayley, then
    \[
        a\in N_l(A)\iff a\in N_m(A)\iff a\in N_r(A).
    \]\label{lem:1v}
    \item If \(a\in N(A)\), then \(a\) is a \(*\)-element if and only if \(a\) is Cayley.\label{lem:1vi}
    \item If \((xy)a=y(xa)\) for all \(x,y\in A\), then
    \[
        a\in N_r(A)
        \iff [A,A]a=0
        \iff x(ya)=y(xa)
        \qquad \text{for all } x,y\in A.
    \]\label{lem:1vii}
    
    \item If $(ax)y=a(yx)$ for all \(x,y\in A\), then
    \[
        a\in N_l(A)
        \iff a[A,A]=0
        \iff (ax)y=(ay)x
        \qquad \text{for all } x,y\in A.
    \]\label{lem:1vii'}
    
    \item \(a\) is a skew \(*\)-element with
\(
x(ya)=y(xa)
\text{ for all } x,y\in A
\)
if and only if
\[
y(xa^*)=-x(ay^*)
\qquad \text{for all } x,y\in A.
\] \label{lem:1viii}
    \item  If \(a\) is a skew \(*\)-element, then \(2a^2=0\). In particular, if \(A\) is \(2\)-torsion-free, then \(a^2=0\) for all skew \(*\)-elements. \label{lem:1ix}
\end{enumerate}
\end{lemma}

\begin{proof}
    \ref{lem:1i}: If \(a\in A\) is a skew \(*\)-element, then \((ax)^*=x^*a^*=x^*(-a)=-x^*a=-ax\) for all \(x\in A\).  If \((ax)^*=-ax\) for all \(x\in A\), then \(a^*=-a\), so \((xa)^*=a^*x^*=-ax^*=(ax^*)^*=xa^*=x(-a)=-xa\). If \((xa)^*=-xa\) for all \(x\in A\), then \(a^*=-a\), so \(x^*a=-(x^*a)^*=-a^*x=ax\); hence \(a\) is a skew \(*\)-element.\\

    \noindent\ref{lem:1ii}: The first implication follows from the definition of a Cayley element. The second implication follows from \(a1\) being a \(*\)-element. \\
    
    \noindent\ref{lem:1iii}: Suppose that \(a\in A\) is Cayley. Then \(a\) is a \(*\)-element, so \(a(xy)=(xy)^*a=(y^*x^*)a\) for all \(x,y\in A\). Since \(a\) is Cayley, \((y^*x^*)a=x^*(y^*a)\). Since all \(ax\) are \(*\)-elements, \(a\) is in particular a \(*\)-element, and so \(x^*(y^*a)=x^*(ay)=(ay)x\). Conversely, if \(a(xy)=(ay)x=x^*(ay)\) for all \(x,y\in A\), the second equality says that all \(ay\) are \(*\)-elements. It then follows from the first equality that \(x^*(y^*a)=x^*(ay)=(ay)x=a(xy)=(xy)^*a=(y^*x^*)a\).\\
    
    \noindent\ref{lem:1iv}: Suppose that all \(ax\) are  \(*\)-elements. Then, \(a(xy)=(xy)^*a=(y^*x^*)a\) and \((ax)y=y^*(ax)=y^*(x^*a)\) for all \(x,y\in A\). Hence, \(a\in N_l(A)\iff a\in N_r(A)\).\\

    \noindent\ref{lem:1v}: Assume that \(a\in A\) is Cayley. By \ref{lem:1iv}, \(a\in N_l(A)\iff a\in N_r(A)\). If \(a\in N_l(A)\), then \((xa)y=(ax^*)y=a(x^*y)=x(ay)\) for all \(x,y\in A\), so \(a\in N_m(A)\). Conversely, if \(a\in N_m(A)\), then \(a(xy)=x^*(ay)=(x^*a)y=(ax)y\) for all \(x,y\in A\), so \(a\in N_l(A)\). \\
    
    \noindent\ref{lem:1vi}: Suppose that \(a\in N(A)\) is a \(*\)-element. Then, \((xy)a=x(ya)=x(ay^*)=(xa)y^*\) for all \(x,y\in A\). Moreover, \((xa)y^*=(ax^*)y^*=a(x^*y^*)=a(yx)^*=(yx)a=y(xa)\), so \(a\) is Cayley. By \ref{lem:1ii}, the converse is always true.\\

   \noindent\ref{lem:1vii}: Suppose that \((xy)a=y(xa)\) for all \(x,y\in A\). If \(a\in N_r(A)\), then \((xy)a=x(ya)=(yx)a\) for all \(x,y\in A\), so \([A,A]a=0\). If \([A,A]a=0\), then \(x(ya)=(yx)a=(xy)a=y(xa)\) for all \(x,y\in A\). Finally, if \(x(ya)=y(xa)\) for all \(x,y\in A\), then \((xy)a=y(xa)=x(ya)\), so \(a\in N_r(A).\)\\

   \noindent\ref{lem:1vii'}: The proof is similar to that of \ref{lem:1vii}.\\

    \noindent\ref{lem:1viii}: If \(a\in A\) is a skew \(*\)-element with \(x(ya)=y(xa)\) for all \(x,y\in A\), then \(y(xa^*)=-y(xa)=-x(ya)=-x(ay^*)\) for all \(x,y\in A\). Conversely, if \(y(xa^*)=-x(ay^*)\) for all \(x,y\in A\), then by setting \(x=1\), we see that \(a\) is a skew \(*\)-element and \(x(ya)=-y(a^*x^*)=-y(xa)^*=y(xa)\).\\
    
    \noindent\ref{lem:1ix}: Let \(a\in A\) be a skew \(*\)-element, i.e.,  \(a^*=-a\) and \(ax=x^*a\) for all \(x\in A\). Setting \(x=a\), we see that \(a^2=a^*a=-a^2\). Hence \(2a^2=0\), so if \(A\) is  \(2\)-torsion-free, then \(a^2=0\).
\end{proof}

\begin{remark}
    In \autoref{lem:1}, \ref{lem:1ii} slightly extends \cite[(i) of Proposition~5.3]{McC85}, while \ref{lem:1vii} generalizes the combination of \cite[(vi) of Proposition~5.3 and (vi) of Theorem~5.5]{McC85}.
\end{remark}

\subsection{Classification of involutions}
In this subsection, we classify all \(A\)-involutions of \(\Cay(A,\mu)\), where \(A\) is a \(*\)-algebra and \(\mu\in K\setminus\{0\}\). For related work on involutions of composition algebras, see, e.g., ~\cite{Pum03}.

For each \(x\in A\), the elements \(xx^*\) and \(x+x^*\) are called the \emph{norm} and \emph{trace of \(x\)}, respectively.

\begin{theorem}\label{thm:inv}
Let \(A\) be a \(*\)-algebra and let \(\mu\in K\setminus\{0\}\). A map 
\[\gamma\colon \Cay(A,\mu)\to\Cay(A,\mu)\] is an \(A\)-involution of \(\Cay(A,\mu)\) if and only if it is of the form

\begin{equation}\label{eq:gamma_form}
\gamma(x,y)=(x^*+ay^*,by)
\qquad (x,y\in A),
\end{equation}
for some \(a,b\in A\) satisfying

\begin{align}
    a\in N(A) \quad & \text{and} \quad a \text{ is a skew \(*\)-element}, \label{eq:inv_a}\\
    b\in N_l(A) \quad & \text{and} \quad b^2=1, \label{eq:inv_b}\\
    a&=ba, \label{eq:inv_3}\\
    \mu xy&=a^2(xy)+\mu(xb^*)(by) \qquad\text{for all \(x,y\in A\).} \label{eq:inv_4}
\end{align}
Moreover, the norm and trace induced by \(\gamma\) are given by

\begin{align}
(x,y)\gamma(x,y)
&=\left(xx^*+xay^*+\mu(by)^*y,(b+1)yx-y^2a\right),\label{eq:norm}\\
(x,y)+\gamma(x,y)
&=\left(x+x^*+ay^*,(b+1)y\right),
\end{align}
for all \(x,y\in A\).
\end{theorem}

\begin{proof}
Let \(\gamma\) be an \(A\)-involution of \(\Cay(A,\mu)\) and suppose that \(
\gamma(0,1)=\left(a,b\right)
\) for some \(a,b\in A\). Since \((x,y)=(x,0)+(y,0)(0,1)\) for all \(x,y\in A\), and since \(\gamma\) is an anti-automorphism, we obtain
\[
\gamma(x,y)
= \gamma(x,0)+\gamma\left((y,0)(0,1)\right)
= \left(x^*,0\right)+\left(a,b\right)\left(y^*,0\right)
= \left(x^*+ay^*, by\right)
\]
for all \(x,y\in A\). Hence,
\begin{align*}
\gamma^2(x,y)
&=\gamma\left(x^*+ay^*,by\right)
=\left(\left(x^*+ay^*\right)^*+a(by)^*, b(by)\right)\\
&=\left(x+ya^*+a(by)^*, b(by)\right),
\end{align*}
so the condition \(\gamma^2(x,y)=(x,y)\) for all \(x,y\in A\) is equivalent to 
\begin{align}
    a(bx)^*&=-xa^*, \label{eq:p-inv_01}\\
    b(bx)&=x,\label{eq:p-inv_02}
\end{align}
for all \(x\in A\). 

Next, let \(u,v,x,y\in A\). On the one hand,
\begin{align*}
\gamma\left((u,v)(x,y)\right)
&=\gamma\left(ux+\mu y^*v, yu+vx^*\right)\\
&=\left(\left(ux+\mu y^*v\right)^*+a\left(yu+vx^*\right)^*, b\left(yu+vx^*\right)\right)\\
&=\left(x^*u^*+\mu v^*y+a\left(u^*y^*\right)+a\left(xv^*\right), b\left(yu\right)+b\left(vx^*\right)\right).
\end{align*}
On the other hand,
\begin{align*}
\gamma(x,y)\gamma(u,v)
&=(x^*+ay^*, by)(u^*+av^*, bv)\\
&=\left(x^*u^*+x^*(av^*)+(ay^*)u^*+(ay^*)(av^*)+\mu(v^*b^*)(by),\right.\\
&\phantom{=\ \ }\left.(bv)x^*+(bv)(ay^*)+(by)x+(by)(va^*)\right).
\end{align*}
By comparing components, the condition \(
\gamma((u,v)(x,y))=\gamma(x,y)\gamma(u,v)
\) for all \(u,v,x,y\in A\) is equivalent to the following two equalities:
\begin{align*}
\mu v^*y + a(u^*y^*) + a(xv^*)
&=x^*(av^*)+(ay^*)u^*+(ay^*)(av^*)+\mu(v^*b^*)(by),\\
b(yu)+b(vx^*)
&=(bv)x^*+(bv)(ay^*)+(by)u+(by)(va^*),
\end{align*}
for all \(u,v,x,y\in A\).

We claim that these two equalities are equivalent to 
\begin{align}
    a(xy)&=(ay)x=x^*(ay), \label{eq:p-inv_1}\\
    \mu xy&=(ay^*)(ax)+\mu(xb^*)(by), \label{eq:p-inv_2}\\
    b(xy)&=(bx)y, \label{eq:p-inv_3} \\
    (by)(xa^*)&=-(bx)(ay^*), \label{eq:p-inv_4}
\end{align}
for all \(x,y\in A\). Indeed, setting \(y=0\) and, separately, \(v=0\) in the first equality yields \eqref{eq:p-inv_1}. Setting \(u=x=0\) in the first equality yields \eqref{eq:p-inv_2}. Similarly, setting \(y=0\) in the second equality yields \eqref{eq:p-inv_3}, and setting \(u=x=0\) in the second equality yields \eqref{eq:p-inv_4}. Conversely, adding \eqref{eq:p-inv_1} and \eqref{eq:p-inv_2} recovers the first equality, and adding \eqref{eq:p-inv_3} and \eqref{eq:p-inv_4} recovers the second equality.

Now, we show that \eqref{eq:p-inv_01}--\eqref{eq:p-inv_4} are equivalent to \eqref{eq:inv_a}--\eqref{eq:inv_4}.
First, assume that \eqref{eq:p-inv_01}--\eqref{eq:p-inv_4} hold. We note that \eqref{eq:p-inv_3} says precisely that \(b\in N_l(A)\). Under this condition, \(b(bx)=x\) is equivalent to \(b^2=1\), so \eqref{eq:inv_b} holds. These two properties of \(b\) also imply that \eqref{eq:p-inv_4} is equivalent to \(y(xa^*)=-x(ay^*)\) for all \(x,y\in A\), which is equivalent to \(a\) being a skew \(*\)-element with \(x(ya)=y(xa)\) for all \(x,y\in A\) by \ref{lem:1viii} of \autoref{lem:1}, hence \(a\) is a skew \(*\)-element. Moreover, since \(a\) is also Cayley (by \eqref{eq:p-inv_1} and \ref{lem:1iii} of \autoref{lem:1}), \ref{lem:1vii} of \autoref{lem:1} implies that \(a\in N_r(A)\), and so, by \ref{lem:1v} of \autoref{lem:1}, \(a\in N(A)\). Now, to get \eqref{eq:inv_4}, it suffices to note that since \(a\in N(A)\) is Cayley, \((ay^*)(ax)=a(y^*(ax))=a((ax)y)=a^2(xy)\). Finally, setting \(x=1\) in \eqref{eq:p-inv_01} and using the fact that \(a\) is a skew \(*\)-element yields \eqref{eq:inv_3}.

Conversely, assume that \eqref{eq:inv_a}--\eqref{eq:inv_4} hold.
Then, \eqref{eq:p-inv_02} and \eqref{eq:p-inv_3} follow directly. Moreover, \eqref{eq:inv_a} implies that \(a\) is Cayley and in \(N(A)\) (by \ref{lem:1vi} of  \autoref{lem:1}), hence \eqref{eq:p-inv_1} holds, and \eqref{eq:inv_4} implies \eqref{eq:p-inv_2}. Combining \ref{lem:1vii} and \ref{lem:1viii} of \autoref{lem:1} also implies that \(y(xa^*)=-x(ay^*)\) for all \(x,y\in A\), which we have shown is equivalent to \eqref{eq:p-inv_4}. We see that \eqref{eq:p-inv_01} follows from \(a\) being Cayley, skew, and in \(N(A)\) (\ref{lem:1vii} of \autoref{lem:1}), and equality \eqref{eq:inv_3}: for all \(x\in A\), \(a(bx)^*=(bx)a=(xb)a=x(ba)=xa=-xa^*\).

Finally, given \(a,b\in A\) satisfying \eqref{eq:inv_a}--\eqref{eq:inv_4}, define \(\gamma\) by \eqref{eq:gamma_form}. Then \(\gamma\) is \(K\)-linear and, by the preceding equivalences, it is an \(A\)-involution of \(\Cay(A,\mu)\).
\end{proof}

\begin{corollary}\label{cor:alpha-beta-involutions}
Let \(A\) be a \(*\)-algebra and let \(\mu\in K\setminus\{0\}\). Then the maps \[\alpha,\beta\colon\Cay(A,\mu)\to\Cay(A,\mu)\] defined by
\[
\alpha(x,y)\colonequals(x^*,-y)\quad\text{and}\quad \beta(x,y)\colonequals(x^*,y)\qquad (x,y\in A)
\]
are \(A\)-involutions of \(\Cay(A,\mu)\). Moreover, if \(A\) has no zero divisors, then \(\alpha\) and \(\beta\) are the only \(A\)-involutions of \(\Cay(A,\mu)\).
\end{corollary}

\begin{proof}
    The maps \(\alpha\) and \(\beta\) are \(A\)-involutions of \(\Cay(A,\mu)\) since \(a=0\) and \(b=\pm 1\) satisfy \eqref{eq:inv_a}--\eqref{eq:inv_4}.
    Now, suppose that \(A\) has no zero divisors. From \autoref{thm:inv}, \(b^2=1\), so \((b-1)(b+1)=0\), implying that \(b=\pm 1\). It follows that \(b^*=b\) and that \(b^*b=1\), hence setting \(x=y=1\) in \eqref{eq:inv_4} yields \(\mu=a^2+\mu\), or equivalently \(a^2=0\). Since A has no zero divisors, it has no nonzero nilpotent elements; hence \(a=0\). 
\end{proof}

\begin{remark}\label{rem:alpha-beta-equality}
    Note that \(\alpha=\beta\) if and only if \(2A=0\).
\end{remark}

In the following example, we determine all involutions of the split octonion algebra extending the standard involution of the split quaternion algebra.

\begin{example}
Let \(K\) be a field and let \(A=M_2(K)\), the split quaternion algebra over \(K\). Its standard involution is the matrix adjugate (see, e.g., \cite[Chapter 1.8]{SV00}).

We determine all \(A\)-involutions of \(\Cay(A,1)\), the split octonion algebra. Let \(\gamma\) be such an involution and let \(a,b\in A\) be as in \autoref{thm:inv}. Since \(A\) is associative, \ref{lem:1i} and \ref{lem:1vi} of \autoref{lem:1} imply that \(a\) is Cayley. Hence, by \ref{lem:1vii} of \autoref{lem:1}, we have \([A,A]a=0\). Now,
\[
\begin{pmatrix}
-1 & 1\\
0 & 1
\end{pmatrix}
=
\begin{pmatrix}
1 & 0\\
1 & 0
\end{pmatrix}
\begin{pmatrix}
0 & 1\\
0 & 0
\end{pmatrix}
-
\begin{pmatrix}
0 & 1\\
0 & 0
\end{pmatrix}
\begin{pmatrix}
1 & 0\\
1 & 0
\end{pmatrix},
\]
and the matrix on the right-hand side belongs to \([A,A]\), while the matrix on the left-hand side is invertible. Hence \(a=0\), and consequently
\(
\gamma(x,y)=(x^*,by)\)
for all \(x,y\in A\), where \(b^2=b^*b=1\). In particular, \(b=b^*bb=b^*\). Writing
\[
b=
\begin{pmatrix}
b_{11} & b_{12}\\
b_{21} & b_{22}
\end{pmatrix}\quad\text{and}\quad
b^*=
\begin{pmatrix}
b_{22} & -b_{12}\\
-b_{21} & b_{11}
\end{pmatrix},
\]
we have that \(b=b^*\) if and only if \(b_{11}=b_{22}\) and \(2b_{12}=2b_{21}=0\). Therefore,
\[
b^2=
\begin{pmatrix}
b_{11}^2+b_{12}b_{21} & 0\\
0 & b_{11}^2+b_{12}b_{21}
\end{pmatrix},
\]
so \(b^2=1\) if and only if \(b_{11}^2+b_{12}b_{21}=1\).

If \(\character K\neq2\), then \(b_{12}=b_{21}=0\) and \(b_{11}=b_{22}=\pm1\). Hence \(\gamma\) is either \(\alpha\) or \(\beta\). If \(\character K=2\), however, any \(b_{11},b_{12},b_{21}\in K\) satisfying \(b_{11}^2+b_{12}b_{21}=1\) yields such an involution. In particular, additional involutions beside \(\alpha\) and \(\beta\) occur in characteristic \(2\).
\end{example}

The following proposition shows that Conditions \eqref{eq:inv_b} and \eqref{eq:inv_4} can be replaced by simpler ones under mild hypotheses. 

\begin{proposition}\label{prop:inv-2tf}
Let \(A\) be a \(2\)-torsion-free \(*\)-algebra, \(\mu\in K\setminus\{0\}\) be cancellable, and suppose that \(a,b\in A\) satisfy \eqref{eq:inv_a} and \eqref{eq:inv_3}. Then \eqref{eq:inv_b} and \eqref{eq:inv_4} are equivalent to the following conditions:
\begin{equation}
b\in N(A), \quad b^2=1,\quad \text{and
}\quad b^*=b.\label{eq:inv_2tf_2}
\end{equation}
\end{proposition}

\begin{proof}
    Assume first that \eqref{eq:inv_a}--\eqref{eq:inv_4} hold. 
    We show that under our additional assumptions, \(b\in N(A)\) and \(b=b^*\). First, \(a\) being a skew \(*\)-element implies that \(a^2=0\) (see \ref{lem:1ix} of \autoref{lem:1}). Using in addition the cancellability of \(\mu\), \eqref{eq:inv_4} is equivalent to
    \begin{equation}\label{eq:inv_2tf_4}
        xy=(xb^*)(by)\qquad\text{for all }x,y\in A.
    \end{equation}
     Setting \(x=1\) and \(y=b\) in this equality yields \(b=b^*b^2=b^*\). In particular, it follows that \(b\) is in the right nucleus too. Finally, setting \(x=xb\) in \eqref{eq:inv_2tf_4} yields that \(b\) is in the middle nucleus. 
     The converse is straightforward since, when \eqref{eq:inv_2tf_2} holds, \eqref{eq:inv_2tf_4} is trivial and \(a^2=0\) by \ref{lem:1ix} of \autoref{lem:1}.
\end{proof}

The next example describes the smallest generic \(*\)-algebras over \(K\) generated by \(a\) and \(b\) satisfying \eqref{eq:inv_a}, \eqref{eq:inv_3}, and \eqref{eq:inv_2tf_2};
as \(K\)-modules, they are free of rank \(3\).

\begin{example}
Let \(A\) be a \(2\)-torsion-free \(*\)-algebra and let \(\mu\in K\setminus\{0\}\) be cancellable. Suppose that \(a,b \in A\) satisfy \eqref{eq:inv_a}, \eqref{eq:inv_3}, and \eqref{eq:inv_2tf_2}. 
Since \(a,b \in N(A)\), the \(*\)-subalgebra \(B\) of \(A\) generated by \(a\) and \(b\) is associative. Moreover, \eqref{eq:inv_a}, \eqref{eq:inv_3}, and \eqref{eq:inv_2tf_2} imply the following equalities:
\[a^2=0,\quad a=ab,\quad ab=ba,\quad b^2=1,\quad a^*=-a,\quad b^*=b.\]
Hence \(B\) is a quotient of \(K[a,b]/(a^2,a-ab,b^2-1)\), 
where the involution is determined by \(a^*=-a\) and \(b^*=b\). As a \(K\)-module, \(B\) has rank at most \(3\) since 
\(K[a,b]/(a^2,a-ab,b^2-1)\cong K\oplus Ka\oplus Kb\). 
\end{example}

Motivated by \autoref{thm:inv}, we introduce the notation \(S(A)\) for the set of all skew \(*\)-elements of \(A\), where \(A\) is a \(*\)-algebra. 

\begin{proposition}\label{prop:extension-I(A)}
    Let \(A\) be a \(*\)-algebra and suppose that \(\mu\in K\setminus\{0\}\) is cancellable. If \(S(A) = 0\), then \(S(\Cay_\gamma(A,\mu)) = 0\)
    for all \(A\)-involutions \(\gamma\) of \(\Cay(A,\mu)\). 
\end{proposition}

\begin{proof}
    Let \(\gamma\) be an \(A\)-involution of \(\Cay(A,\mu)\). It follows from \autoref{thm:inv} that there exist a skew \(*\)-element \(a\in A\) and an element \(b\in A\) such that 
    \[
    \gamma(x,y)=(x^*+ay^*,by)
            \qquad \text{for all } x,y\in A. 
    \]
    Since \(S(A)=0\), we have \(a=0\). 
    Let \((u,v)\in S(\Cay_\gamma(A,\mu))\) be arbitrary. We need to show that \(u=v=0\). For all \((x,y)\in \Cay(A,\mu)\), we have \(\gamma\left((u,v)(x,y)\right) = -(u,v)(x,y)\), which implies
    \[
	\gamma(ux+\mu y^*v,vx^*+y^*u) = -(ux+\mu y^*v,vx^*+y^*u)\qquad\text{for all }u,v,x,y\in A.
    \]
    Using the above formula for \(\gamma\), we obtain 
    \[
	\left((ux+\mu y^*v)^*,b(vx^*+y^*u)\right) = -(ux+\mu y^*v,vx^*+y^*u)\qquad\text{for all }u,v,x,y\in A.
    \]
    Setting \(y=0\) yields \(\left((ux)^*,b(vx^*)\right)=\left(-(ux),-vx^*\right)\). In particular, \(u\) is a skew \(*\)-element of \(A\), hence \(u =0\). Similarly, setting \(x=0\) we have \(\left(\mu( y^*v)^*,b(y^*u)\right) = (-\mu y^*v,-y^*u)\). Since \(\mu\) is cancellable, we conclude that \(v\in S(A)\), hence \(v=0\), which completes the proof. 
\end{proof}

An immediate consequence of \autoref{prop:extension-I(A)} is that for a \(*\)-algebra \(A\) satisfying \(S(A) = 0\), the Cayley--Dickson construction starting with \(A\) always produces \(*\)-algebras \(B\) satisfying \(S(B) = 0\) provided that all the \(\mu\)'s used in the process are cancellable. This holds, for instance, if we start with a field \(K\) (see, e.g., \autoref{prop:cay-over-R} in \autoref{sec:*-isomorphisms}).

\subsection{Nuclear, commuting, central, and scalar involutions}
If \(A\) is a \(*\)-algebra, then \(*\) is called \emph{nuclear}, \emph{commuting}, \emph{central} or \emph{scalar} if all norms and all traces lie in \(N(A)\), \(C(A)\), \(Z(A)\) or \(K\), respectively. Since \(A\) is assumed to be unital, all traces are built from norms: \[x+x^*=x1^*+1x^*=(x+1)(x+1)^*-xx^*-1\cdot 1^*\qquad\text{for all }x\in A.\] In particular, if all norms lie in \(N(A)\), \(C(A)\), \(Z(A)\) or \(K\), respectively, then so do all traces.

In \cite[(viii)--(xii) of Theorem 6.8]{McC85}, McCrimmon described the nuclei, commuter, and center of \(\Cay(A,\mu)\), where \(A\) is a \(*\)-algebra and \(\mu\in K\setminus\{0\}\) is cancellable. The relevant arguments, however, essentially carry over without the cancellability assumption. The next proposition records the resulting slightly more general statements, following McCrimmon's arguments. Following \cite{McC85}, for any \(x\in A\), we write \(\Sk(x)\colonequals x-x^*\). With this notation, we obtain the following two propositions: 

\begin{proposition}\label{prop:ncz}
    Let \(A\) be a \(*\)-algebra and let \(\mu\in K\setminus\{0\}\). Then the following hold:
    \begin{enumerate}[label={(\roman*)}]        
        \item \(\begin{aligned}[t]
             N_l(\Cay(A,\mu))&=\{(u,v)\ | \ u\in Z(A),\, \mu v\in N_m(A)\cap C(A),\\  &\phantom{=\ \ }\text{ with } \, v(xy)=(vy)x \, \text{ for all }\, x,y\in A \}.
         \end{aligned}\)\label{it:ncz-nl}
         
        \item \(\begin{aligned}[t]
            N_m(\Cay(A,\mu)&=\{ (u,v) \ | \ u\in N_m(A)\cap C(A), \, \mu v\in C(A), \\
            &\phantom{=\ \ }\text{ with } \, (vx)y=(vy)x, \, \mu y(xv)=\mu x(yv) \, \text{ for all } x,y\in A\}.
        \end{aligned}\)\label{it:ncz-nm}

        \item \(\begin{aligned}[t]
                N_r(\Cay(A,\mu))&=\{ (u,v)\ | \ u\in Z(A),\, \mu v\in N_m(A)\cap C(A), \\
                &\phantom{=\ \ }\text{ with} \, (xy)v=y(xv) \, \text{ for all }\, x,y\in A \}.
            \end{aligned}\)\label{it:ncz-nr}
            
        \item \(\begin{aligned}[t]
            N(\Cay(A,\mu))&=\{(u,v)\ | \ u, \mu v\in Z(A), \, v\in N_l(A), \\
            &\phantom{=\ \ }\text{ with }\, v[A,A]=0 \text{ and } (xy)v=y(xv) \text{ for all } x,y\in A\}.
        \end{aligned}\)\label{it:ncz-n}        

        \item \(\begin{aligned}[t]
            C(\Cay(A,\mu))&=\{(u,v)\ | \ u,\mu v\in C_*(A), \, \text{with }\, v\Sk(A)=0 \}.
        \end{aligned}\)\label{it:ncz-c}
        
        \item \(\begin{aligned}[t]
            Z(\Cay(A,\mu))&=\{(u,v)\ |\ u,\mu v \in Z_*(A),\, v\in N_l(A),\\ 
            &\phantom{=\ \ }\text{ with }\, v\Sk(A)=0 \text{ and } (xy)v=y(xv) \text{ for all } x,y\in A\}.
        \end{aligned}\)\label{it:ncz-z}
    \end{enumerate}
\end{proposition}

\begin{proof}
\noindent\ref{it:ncz-nl}--\ref{it:ncz-nr}: These equalities follow from~\cite[(6.8)]{McC85}. \\

\noindent\ref{it:ncz-n}: Since the nucleus is the intersection of the left, middle, and right nuclei, \((u,v)\in N(\Cay(A,\mu))\) if and only if \(u\in Z(A), \ \mu v\in N_m(A)\cap C(A)\), and 
    \[
        v(xy)=(vy)x, \ (vx)y=(vy)x, \ \mu y(xv)=\mu x(yv),\ (xy)v=y(xv) \qquad \text{for all }x,y\in A.
    \]
    Now, note that the first two equalities are equivalent to \(v\in N_l(A)\) and \(v[A,A]=0\). Moreover, the third equality follows from the first equality and the fact that we have \(\mu v\in N_m(A)\cap C(A)\). Indeed, assuming that these relations hold yields
    \[\mu y(xv)=\mu y(vx)=\mu (yv)x=\mu (vy)x=\mu (vx)y=\mu x(yv)\qquad\text{for all }x,y\in A.\]

\noindent\ref{it:ncz-c}: This equality follows directly from~\cite[(6.7)]{McC85}.\\
    
\noindent\ref{it:ncz-z}: This equality follows by combining \ref{it:ncz-n} and \ref{it:ncz-c}, and noting that if \(v\in N_l(A)\) and \(v\Sk(A)=0\), then \(v[A,A]=0\); this since for all \(x,y\in A\),
    \[        v(xy)=(vx)y=(vx^*)y=v(x^*y)^*=v(y^*x)=(vy^*)x=(vy)x=v(yx).\qedhere
    \]
\end{proof}

\begin{proposition}\label{prop:nuclear-inv}
    Let \(A\) be a \(*\)-algebra and let \(\mu\in K\setminus\{0\}\). A map 
    \[\gamma\colon \Cay(A,\mu)\to \Cay(A,\mu)\] is a nuclear, commuting, central or scalar \(A\)-involution of \(\Cay(A,\mu)\) if and only if it is of the form
    \[
    \gamma(x,y)=(x^*+ay^*,by)
    \qquad (x,y\in A),
    \]
    for some \(a,b\in A\) satisfying \eqref{eq:inv_a}--\eqref{eq:inv_4}, together with these corresponding conditions:
    \begin{enumerate}[label=(\roman*)]
        \item \textbf{Nuclear case.} The involution \( *\) is central and for all \(x\in A\),
        \[
        ax,\, \mu (bx)^*x,\, \mu (b+1)x \in Z(A) \quad \text{and} \quad (b+1)x\in N_l(A), 
        \]
        with \((b+1)[A,A]=0\) and \((uv)((b+1)x)=v(u(b+1)x)\) for all \(u,v\in A\).\label{it:nuclear-inv}

        \item \textbf{Commuting case.}
        The involution \(*\) is commuting and for all \(x\in A\),
        \[ax, \, \mu (bx)^*x,\, \mu(b+1)x\in C_*(A) \quad \text{and} \quad (b+1)A\Sk(A)=0. \]\label{it:commuting-inv}\vspace{-3ex}
        
        \item \textbf{Central case.} The involution \( * \) is central and for all \(x\in A\),
        \[
        ax,\, \mu (bx)^*x,\, \mu (b+1)x \in Z_*(A) \quad \text{and} \quad (b+1)x\in N_l(A),
        \]
        with \((b+1)A\Sk(A)=0\) and \((uv)((b+1)x)=v(u(b+1)x)\) for all \(u,v\in A\).\label{it:central-inv}
        
        \item \textbf{Scalar case.} The involution \(*\) is scalar and \(\gamma=\alpha\).\label{it:scalar-inv}
    \end{enumerate}
\end{proposition}

\begin{proof}
    As stated in \eqref{eq:norm} of \autoref{thm:autom}, the norm of \((x,y)\in \Cay(A,\mu)\) induced by \(\gamma\) is given by \[(x,y)\gamma(x,y)=\left(xx^*+xay^*+\mu(by)^*y,(b+1)yx-y^2a\right).\]
    We use \autoref{prop:ncz} to get the desired characterizations.

    \noindent \ref{it:nuclear-inv}: The involution $\gamma$ is nuclear if and only if \((x,y)\gamma(x,y)\in N(\Cay(A,\mu))\) for all \(x,y\in A\). By \ref{it:ncz-n} of \autoref{prop:ncz}, this is equivalent to the following conditions holding for all \(x,y\in A\):
    \begin{align*}
        xx^*+xay^*+\mu (by)^*y,\ \mu\left((b+1)yx-y^2a\right)&\in Z(A),\\
        (b+1)yx-y^2a&\in N_l(A),\\
        \left((b+1)yx-y^2a\right)[A,A]&=0,\\
        (uv)\left((b+1)yx-y^2a\right)&=v\left(u\left((b+1)yx-y^2a\right)\right), 
    \end{align*}
for all \(u,v\in A\). Now, by setting \(x=0\) and \(y=0\) separately, we see that these sums satisfy the required properties if and only if each of their terms does. Moreover, since \(a\in N(A)\) is a \(*\)-element, \(xay^*=ax^*y^*\in Z(A)\) for all \(x,y\in A\) if and only if \(ax\in Z(A)\) for all \(x\in A\). The latter condition implies, in particular, that \(\mu x^2a=\mu a(x^*)^2\in Z(A)\) for all \(x\in A\). Note also that $x\in Z(A)$ if and only if $x^*\in Z(A)$. Furthermore, \(a\) is a \(*\)-element in \(N(A)\), so by \ref{lem:1vii}--\ref{lem:1vii'} of \autoref{lem:1}, \([A,A]a=a[A,A]=0\), and so \(ax[A,A]=x^*a[A,A]=0\) for all \(x\in A\). Using in addition that \(ax\in Z(A)\) for all \(x\in A\), we get that, for all \(u,v\in A\),
    \[
    (uv)\left(x^2a\right)=((uv)a)x^2=((vu)a)x^2=(v(ua))x^2=v\left(u\left(ax^2\right)\right)=v\left(u\left(x^2a\right)\right).
    \]
    Hence \(\gamma\) is nuclear if and only if \(*\) is central and for all \(x\in A\), 
    \begin{align*}
        &ax,\ \mu (bx)^*x, \ \mu (b+1) x \in Z(A),\quad (b+1)x\in N_l(A),\\
    &(uv)((b+1)x)=v(u(b+1)x) \qquad\text{for all }u,v\in A,
    \end{align*}
    with \((b+1)A[A,A]=0\). 
    We only have left to show that \((b+1)[A,A]=0\) implies \((b+1)A[A,A]=0\). However, \(b+1\in N_l(A)\), so if \((b+1)[A,A]=0\), then for any \(x,y,z\in A\), 
    \begin{align*}
        ((b+1)x)(yz)&=(b+1)(x(yz))=(b+1)((yz)x)=((b+1)(yz))x\\
        &=((b+1)(zy))x=(b+1)((zy)x)=(b+1)(x(zy))=((b+1)x)(zy).
    \end{align*} 
    
    \noindent \ref{it:commuting-inv}: The involution \(\gamma\) is commuting if and only if \((x,y)\gamma(x,y)\in C(\Cay(A,\mu))\) for all \(x,y\in A\). By \ref{it:ncz-c} of \autoref{prop:ncz}, this is equivalent to the following conditions holding for all \(x,y\in A\):
    \begin{align*}    
    xx^*+xay^*+\mu(by)^*y, \ \mu\left((b+1)yx-y^2a\right)&\in C_*(A), \\
    \left((b+1)yx-y^2a\right)\Sk(A)&=0.
    \end{align*}
    Since \(a\in C(A)\) is a \(*\)-element, \(ay=y^*a=ay^*\) for all \(y\in A\). Using that \(a\in N(A)\), we have that for all \(x\in A\), \(x^2ay=x^2ay^*\), so \(x^2a\Sk(A)=0\). This and the remarks made in the proof of \ref{it:nuclear-inv} then yields the desired characterization.\\

    \noindent \ref{it:central-inv}: Here, it suffices to combine \ref{it:nuclear-inv} and \ref{it:commuting-inv}, and note that, as shown in the proof of \ref{it:ncz-z} of \autoref{prop:ncz}, \((b+1)[A,A]=0\) follows from \((b+1)A\Sk(A)=0\) and \(b+1\in N_l(A)\). \\

    \noindent\ref{it:scalar-inv}: Suppose that \(*\) is scalar and \(\gamma=\alpha\). Then \(a=0\) and \(b=-1\), so \((x,y)\gamma(x,y)=(xx^*-\mu y^*y,0)=\left(xx^*-\mu y^*(y^*)^*,0\right)\in K\) for all \(x,y\in A\). Now suppose instead that \(\gamma\) is scalar. Since \(\gamma\) extends \(*\), \(*\) is scalar. Moreover, \((b+1)yx-y^2a=0\) for all \(x,y\in A\). By setting \(x=0\) and \(y=1\), we have \(a=0\). By then setting \(x=y=1\), we have \(b=-1\), so \(\gamma=\alpha\).
\end{proof}

\section{Isomorphisms}\label{sec:isomorphisms} In \cite{Ea90}, Eakin and Sathaye gave a description of the automorphisms and derivations of the Cayley doubles \(A_n\) defined inductively by \(A_0 = K\), where \(K\) is a field of characteristic other than \(2\) or \(3\), \(A_{n+1} = \Cay_\alpha(A_n,\mu_n)\), where \(\mu_n\in K\setminus\{0\}\). They showed that the automorphism group of \(A_{n+1}\) is isomorphic to the direct product of the automorphism group of \(A_{n}\) and a finite group \(G\), which is either \(\mathbb{Z}/2\) or \(S_3\). In this section, we describe those automorphisms of the algebra \(\Cay(A,\mu)\) that extend the identity automorphism of \(A\). 

\subsection{Algebra isomorphisms} 
Let \(A\) be a \(*\)-algebra and let \(\mu_1,\mu_2\in K\setminus\{0\}\). We call an algebra isomorphism \[\phi\colon\Cay(A,\mu_1)\to \Cay(A,\mu_2)\] that extends the identity automorphism of \(A\) an \emph{\(A\)-isomorphism}. 
\begin{theorem}\label{thm:autom}
    Let \(A\) be a \(*\)-algebra and let \(\mu_1,\mu_2\in K\setminus\{0\}\). A map \[\phi\colon\Cay(A,\mu_1)\to \Cay(A,\mu_2)\] is an \(A\)-isomorphism if and only if it is of the form
    \begin{equation}\label{eq:autom_def}
        \phi(x,y)=(x+cy^*,dy) \qquad (x,y\in A),
    \end{equation}
    for some \(c,d\in A\) satisfying 
    \begin{align}
        c\in N(A) \quad &\text{and} \quad c \text{ is a skew \(*\)-element,} \label{eq:autom_1}\\
        d\in N_l(A) \quad &\text{and} \quad x\mapsto dx \text{ is bijective,} \label{eq:autom_2} \\
            \mu_1 xy&=c^2(xy)+\mu_2(xd^*)(dy) \qquad\text{for all \(x,y\in A\).}\label{eq:autom_3}
    \end{align}
\end{theorem}

\begin{proof}
    Let \(\phi\colon\Cay(A,\mu_1)\to\Cay(A,\mu_2)\) be an \(A\)-isomorphism and let \(c,d\in A\) be such that \(\phi(0,1)=(c,d)\). Then, by using that \(\phi\) is an \(A\)-automorphism together with the relation \((0,y)=(y,0)(0,1)\), we have 
    \[
        \phi(x,y)=\phi(x,0)+\phi\left((y,0)(0,1)\right)=(x,0)+(y,0)(c,d)=(x+yc,dy),
    \]
    for all \(x,y\in A\). Now, let us define a map \(\phi_2\colon A\to A\) by \(\phi_2(x)\colonequals dx\) (\(x\in A\)) and show that \(\phi\) is bijective if and only if \(\phi_2\) is bijective. The direct implication is straightforward. To show the converse implication, assume that  \(\phi_2\) is bijective. Then, on the one hand,
    \[
        \phi(x,y)=(0,0)\iff x+yc=0 \text{ and } \phi_2(y)=0 \iff x=y=0,
    \]
    so \(\phi\) is injective. On the other hand, the surjectivity of \(\phi_2\) implies that for all \(y\in A\), there exists \(y'\in A\) such that \(y=dy'\). It follows that \((x,y)=\phi(x-y'c,y')\) for all \(x,y\in A\), so \(\phi\) is surjective too. From now on, we assume that \(\phi\) is bijective.
    
    Next, we show that \(\phi\) is an automorphism if and only if \(c\) and \(d\) satisfy \eqref{eq:autom_1}--\eqref{eq:autom_3}. Since, for any \(u,v,x,y\in A\), 
    
    \begin{align*}
        \phi\left((u,v)(x,y)\right)&=\phi(ux+\mu_1 y^*v,vx^*+yu)\\
        &=\left(ux+\mu_1 y^*v+(vx^*+yu)c,d(vx^*+yu)\right)\\
        &=\left(ux+\mu_1 y^*v+(vx^*)c+(yu)c,d(vx^*)+d(yu)\right),\\
        \phi(u,v)\phi(x,y)&=(u+vc,dv)(x+yc,dy)\\
        &=\left(ux+u(yc)+(vc)x+(vc)(yc)+\mu_2 (dy)^*(dv),\right.\\
        &\phantom{=\ \ } \left.(dv)x^*+(dv)(yc)^*+(dy)u+(dy)(vc)\right),
    \end{align*}
     the identity \(\phi\left((u,v)(x,y)\right)=\phi(u,v)\phi(x,y)\) holds for all \(u,v,x,y\in A\) if and only if 
     
    \begin{align*}
            \mu_1 y^*v+(vx^*)c+(yu)c&=u(yc)+(vc)x+(vc)(yc)+\mu_2(dy)^*(dv),\\
            d(vx^*)+d(yu)&=(dv)x^*+(dv)(yc)^*+(dy)u+(dy)(vc),
    \end{align*}
     hold for all \(u,v,x,y\in A\). Similarly to the proof of \autoref{thm:inv}, these conditions are equivalent to the following conditions, which hold for any \(x,y\in A\):
     
    \begin{align}
            c  &\text{ is Cayley}, \label{eq:p-autom_1}\\
            \mu_1 xy&=(yc)(x^*c)+\mu_2(xd^*)(dy),\label{eq:p-autom_2}\\
            d&\in N_l(A),\label{eq:p-autom_3}\\
            (dy)(xc)^*&=-(dx)(yc).\label{eq:p-autom_4}
    \end{align}
    Moreover, \eqref{eq:p-autom_3} makes \eqref{eq:p-autom_4} equivalent to \(d\left(y(xc)^*\right)=-d\left(x(yc)\right)\). Since \(\phi_2\) is bijective, this equality is equivalent to \(y(xc)^*=-x(yc)\). Setting \(y=1\), by \ref{lem:1i} of \autoref{lem:1}, \(c\) is then a skew \(*\)-element. Hence \(y(xc)^*=-y(xc)\), so \(y(xc)^*=-x(yc)\) is equivalent to \(y(xc)=x(yc)\). Now, by \eqref{eq:p-autom_1}, \(c\) is Cayley, and so \((xy)c=y(xc)\) by definition. By using \ref{lem:1ii} of \autoref{lem:1}, we have \(c\in N_r(A)\), and so by \ref{lem:1v} of the same lemma, \(c\in N(A)\). Since \(c\) is a \(*\)-element, \(yc=cy^*\), which yields \eqref{eq:autom_def}. Finally, because \(c\) is a skew \(*\)-element in \(N(A)\), we have \((yc)(x^*c)=(cy^*)(x^*c)=c(xy)^*c=c^2(xy)\), making \eqref{eq:p-autom_2} equivalent to \eqref{eq:autom_3}. 

    To conclude the proof, we need to show that if \eqref{eq:autom_1}--\eqref{eq:autom_3} hold, then \(c\) is Cayley. It would then follow that \(c\) is a skew \(*\)-element with \(x(yc)=y(xc)\) for all \(x,y\in A\), since \(c\in N(A)\), and so by \ref{lem:1viii} of \autoref{lem:1}, \eqref{eq:p-autom_4} would hold. For that, it suffices to use \ref{lem:1vi} of \autoref{lem:1}: \(c\) is a \(*\)-element in \(N(A)\), so \(c\) is Cayley.
\end{proof}

\begin{remark}
    The similarity between the conditions satisfied by \(a\) and \(b\) in \autoref{thm:inv} and those satisfied by \(c\) and \(d\) in \autoref{thm:autom} can be explained as follows. If \(\gamma(x,y)=(x^*+ay^*,by)\) defines an involution, then \(\beta\circ\gamma(x,y)= (x+ya^*,by)\) is an automorphism.
\end{remark}

\begin{proposition}
    Assume that \(A\) is a \(2\)-torsion-free \(*\)-algebra, \(\mu_1\in K\setminus\{0\}\) is cancellable, and let \(c,d\in A\) where \(c\) satisfies \eqref{eq:autom_1}. Then \eqref{eq:autom_2}--\eqref{eq:autom_3} are equivalent to the following conditions:
    \begin{equation}
        d\in N(A), \quad  x\mapsto dx \text{ is bijective}, \quad \text{and} \quad \mu_2d^*d=\mu_2 dd^*=\mu_1. \label{eq:autom_2tf_d}
    \end{equation}
    Moreover, if \(\mu_2=\mu_1\), then \(c\) and \(d\) satisfy \eqref{eq:autom_2}--\eqref{eq:autom_3} if and only if
    \begin{equation}
        d\in N(A)\quad \text{and}\quad dd^* = d^*d = 1. \label{eq:autom_2tf_eq_d}
    \end{equation}
\end{proposition}

\begin{proof}
    First, suppose that \eqref{eq:autom_1}--\eqref{eq:autom_3} hold. Since \(A\) is \(2\)-torsion-free, from \ref{lem:1ix} of \autoref{lem:1}, it follows that \(c^2=0\). Equality \eqref{eq:autom_3} then becomes
    \begin{equation}\label{eq:autom_3'}
        \mu_1xy=\mu_2 (xd^*)(dy) \qquad \text{for all } x,y\in A.
    \end{equation}
    Setting \(x=y=1\) in this equality yields \(\mu_1=\mu_2d^*d\). Moreover, since \(x\mapsto dx\) is bijective, there is \(d'\in A\) such that \(dd'=1\). Hence, by \eqref{eq:autom_3'}, \(\mu_1=\mu_1dd'=\mu_2(dd^*)(dd')=\mu_2dd^*\). To prove that \(d\in N(A)\), first note that \(d\in N_l(A)\) if and only if \(d^*\in N_r(A)\). Now replace \(x\) by \(x(yd)\) and \(y\) by \(1\) in \eqref{eq:autom_3'} and use that \(d^*\in N_r(A)\) to get \(\mu_1x(yd)=\mu_2\left(\left(x(yd)\right)d^*\right)d=\mu_2\left(x\left(y(dd^*)\right)\right)d=\mu_1(xy)d\). As \(\mu_1\) is cancellable, we see that \(d\in N_r(A)\). In a similar fashion, by replacing \(x\) by \(xd\), we get that \(\mu_1(xd)y=\mu_2\left((xd)d^*\right)(dy)=\mu_1 x(dy)\), which implies that \(d\in N_m(A)\). We conclude that \(d\in N(A)\).

    Conversely, if \eqref{eq:autom_1} and \eqref{eq:autom_2tf_d} hold, then \(\mu_2(xd^*)(dy)=x(\mu_2 d^*d)y=\mu_1xy\).
    
    Now, suppose that \(\mu_1=\mu_2\). Then the equalities in \eqref{eq:autom_2tf_d} are equivalent to \(d^*d=dd^*=1\). Moreover, these equalities and the fact that \(d\in N(A)\) imply that if \(x\in A\), then \(d(d^*x)=x\), and if \(dx=dy\), then \(x=(d^*d)x=d^*(dx)=d^*(dy)=(d^*d)y=y\), so \(x\mapsto dx\) is bijective.
\end{proof}

\begin{proposition}\label{prop:isom-inv}
    Let \(A\) be \(*\)-algebra, \(\mu_1,\mu_2\in K\setminus\{0\}\), and suppose that \[\phi\colon\Cay(A,\mu_1)\to\Cay(A,\mu_2)\] is an \(A\)-isomorphism defined by \[\phi(x,y)=(x+cy^*,dy)\qquad (x,y\in A),\] where \(c,d\in A\) satisfy \eqref{eq:autom_1}--\eqref{eq:autom_3}. Then the inverse \(\phi^{-1}\) of \(\phi\) is given by \[\phi^{-1}(x,y)=\big(x-c\left(d^{-1}\right)^*y^*,d^{-1}y\big)\qquad (x,y\in A),\]
    where \(d^{-1}\in N_l(A)\) is the unique left and the unique right inverse of \(d\).
\end{proposition}

\begin{proof}
    Let \(\psi\colon\Cay(A,\mu_2)\to \Cay(A,\mu_1)\) be a map defined by 
    \[\psi(x,y)=\big(x-c\left(d^{-1}\right)^*y^*,d^{-1}y\big) \qquad (x,y\in A),\]
    where \(d^{-1}\in N_l(A)\) is the unique left and the unique right inverse of \(d\). Note that such an element exists since \(x\mapsto dx\) is bijective, and so \(d\) has a right inverse \(d'\). Moreover, \(d(d'd)=(dd')d=d=d1\), so \(d'd=1\), hence \(d'\) is also a left inverse of \(d\). Now, for any \(x,y\in A\), \(d(d'(xy))=(dd')(xy)=xy\) and \(d((d'x)y)=(d(d'x))y=((dd')x)y=xy\), so by the injectivity of \(x\mapsto dx\), \(d'\in N_l(A)\) too. If \(d''\) is a right inverse of \(d\), then \(d''=(d'd)d''=d'(dd'')=d'\). Hence \(d'\) is a unique left and a unique right inverse of \(d\), which we denote by \(d^{-1}\). 
    
    Continuing, 
    \begin{align*}
        \psi\circ\phi(x,y)&=\big(x+cy^*-c\left(d^{-1}\right)^*(dy)^*,d^{-1}(dy)\big),\\
        \phi\circ\psi(x,y)&=\big(x-c\left(d^{-1}\right)^*y^*+c\left(d^{-1}y\right)^*,d\left(d^{-1}y\right)\big).
    \end{align*}
    Since \(d,d^{-1}\in N_l(A)\), we have \(d^{-1}(dy)=d(d^{-1}y)=y\) for all \(y\in A\). On the other hand, \(c\in N(A)\) is Cayley, so by \ref{lem:1vii} of \autoref{lem:1}, \([A,A]c=0\). Since \(c\) is a \(*\)-element, \([A,A]c=0\) is equivalent to \(c[A,A]=0\). Hence, for all \(y\in A\), we have \(c(d^{-1})^*y^*=cy^*(d^{-1})^*=c(d^{-1}y)^*\). Replacing \(y\) by \(dy\) in this equality yields \(c(d^{-1})^*(dy)^*=c(d^{-1}dy)^*=cy^*\), from which we get the identities \(\phi\circ \psi=\id\) and \(\psi\circ\phi=\id\).
\end{proof}

\subsection{\texorpdfstring{\(*\)}{*}-algebra isomorphisms}\label{sec:*-isomorphisms}
In this subsection, we turn to the question of describing \texorpdfstring{\(*\)}{*}-algebra isomorphisms between the Cayley doubles of \(A\) equipped with \(A\)-involutions. 

\begin{proposition}\label{prop:starisom}
  Let \(A\) be a \(*\)-algebra and let \(\mu_1,\mu_2\in K\setminus\{0\}\).  Suppose that \(\gamma_1,\gamma_2\) are \(A\)-involutions of \(\Cay(A,\mu_1)\) and \(\Cay(A,\mu_2)\), respectively, of the form
\[
\gamma_i(x,y)=(x^*+a_i y^*, b_i y),
\qquad (x,y\in A, \ i=1,2)
\]
where \(a_i,b_i\in A\) satisfy \eqref{eq:inv_a}--\eqref{eq:inv_4}. Then 
\(
\Cay_{\gamma_1}(A,\mu_1)\) and \(
\Cay_{\gamma_2}(A,\mu_2)
\) are \(A\)-isomorphic as \(*\)-algebras if and only if there exist \(c,d\in A\) satisfying
\eqref{eq:autom_1}--\eqref{eq:autom_3} such that
\begin{equation}
    a_1=(1+b_2)c^*+a_2d^*
    \quad\text{and}\quad
    b_1=d^{-1}b_2d,
\end{equation}
where \(d^{-1}\in N_l(A)\) is the unique left and the unique right inverse of \(d\).

In this case, a \(*\)-algebra \(A\)-isomorphism \[\phi\colon \Cay_{\gamma_1}(A,\mu_1)\to \Cay_{\gamma_2}(A,\mu_2)\] is given by
\[
\phi(x,y)=(x+cy^*,dy)
\qquad (x,y\in A).
\]
\end{proposition}

\begin{proof}
    A map \(\phi\colon\Cay_{\gamma_1}(A,\mu_1)\to\Cay_{\gamma_2}(A,\mu_2)\) is a \(*\)-algebra \(A\)-isomorphism if and only if it is an \(A\)-isomorphism \(\Cay(A,\mu_1)\to \Cay(A,\mu_2)\) satisfying \(\phi\circ \gamma_1=\gamma_2\circ \phi\), or equivalently, \(\gamma_1=\phi^{-1}\circ \gamma_2\circ \phi\). Therefore, let \(\phi\colon \Cay_{\gamma_1}(A,\mu_1)\to \Cay_{\gamma_2}(A,\mu_2)\) be an \(A\)-isomorphism defined by 
    \[\phi(x,y)=(x+cy^*,dy)\qquad (x,y\in A),\]
    where \(c,d\in A\) satisfy \eqref{eq:autom_1}--\eqref{eq:autom_3}. Then, by using \autoref{prop:isom-inv}, we want to show that \(\gamma_1(x,y)=(x^*+a_1y^*,b_1y)\) is equal to
    \begin{align*}
        \phi^{-1}\circ \gamma_2\circ\phi(x,y)&=\phi^{-1}\circ\gamma_2(x+cy^*,dy)\\
   &=\phi^{-1}\left((x+cy^*)^*+a_2(dy)^*,b_2(dy)\right)\\
        &=\big(x^*+yc^*+a_2(y^*d^*)-c\left(d^{-1}\right)^*(b_2dy)^*,d^{-1}\left(b_2(dy)\right)\big),
    \end{align*}
    for all \(x,y\in A\). Since \(a_2,c\in N(A)\) are Cayley, by \ref{lem:1vii} of \autoref{lem:1}, \([A,A]a_2=[A,A]c=0\). Since \(a_2\) and \(c\) are \(*\)-elements, \([A,A]a_2=[A,A]c=0\) is equivalent to \(a_2[A,A]=c[A,A]=0\). Moreover, \(a_2\in N_l(A)\), so \(a_2(y^*d^*)=a_2(d^*y^*)=(a_2d^*)y^*\). Similarly, by using that \(d, d^{-1}, b_2\in N_l(A)\), we obtain \(c(d^{-1})^*(b_2dy)^*=cb_2^*y^*\) and \(d^{-1}\left(b_2(dy)\right)=(d^{-1}b_2)(dy)=(d^{-1}b_2d)y\). Hence \(\gamma_1=\phi^{-1}\circ\gamma_2\circ\phi\) if and only if
    \[a_1y^*=(-c+a_2d^*-cb_2^*)y^* \quad\text{and}\quad b_1y=(d^{-1}b_2d)y
    \]
    hold for all \(y\in A\), which is equivalent to the same equalities holding for \(y=1\). By using that \(c\) is a skew \(*\)-element, we arrive at the desired conclusion.
\end{proof}

Recall from \autoref{cor:alpha-beta-involutions} that for any \(*\)-algebra \(A\), the \(A\)-involutions \(\alpha\) and \(\beta\) of \(\Cay(A,\mu)\) are defined by \(\alpha(x,y)\colonequals(x^*,-y)\) and \(\beta(x,y)\colonequals (x^*,y)\) for all \(x,y\in A\). With this notation, we can state the following two results:

\begin{corollary}
    Let \(A\) be a \(*\)-algebra, \(\mu_1,\mu_2\in K\setminus\{0\}\), and suppose that \(\gamma\) is an \(A\)-involution of \(\Cay(A,\mu_1)\). If  \(\Cay_{\gamma}(A,\mu_1)\) and \(\Cay_{\alpha}(A,\mu_2)\) are \(A\)-isomorphic as \(*\)-algebras, then \(\gamma =\alpha\). 
\end{corollary}

\begin{proof}
    Let \(\gamma(0,1) = (a,b)\). Recall that \(\alpha(0,1) = (0,-1)\). By \autoref{prop:starisom}, there exist \(c,d\in A\) satisfying \eqref{eq:autom_1}--\eqref{eq:autom_3} such that 
    \[
        a=(1-1)c^*+0d^* = 0
    \quad\text{and}\quad
    b=d^{-1}(-1)d=-1,
    \]
    from which the result follows immediately. 
\end{proof}
We conclude the paper with a description of the low-dimensional \(*\)-algebras over \(\mathbb{R}\)  that arise from the Cayley–Dickson construction when all admissible involutions are taken into account. 

\begin{proposition}\label{prop:cay-over-R}
Let \(K=A_0=\mathbb{R}\), and for \(n=0,1\) define inductively
\[
A_{n+1}=\Cay_{\gamma_{n+1}}(A_n,\mu_n),
\qquad
\mu_n\in\mathbb{R}\setminus\{0\},
\]
where \(A_n\) is a \(*\)-algebra equipped with an involution \(\gamma_n\). Then, up to \(*\)-algebra isomorphism, the \(*\)-algebras \(A_0,A_1,A_2\) are exactly those in \autoref{tab:Cayley-Dickson} below.

\begin{table}[ht!]
\centering
\[
\begin{array}{c c l l}
\toprule
n & \dim_\mathbb{R}A_n & A_n & \text{Involutions}\\
\midrule
0 & 1 & \mathbb{R}&\text{identity}\\
\midrule
1 & 2 & \mathbb{C}&\mathbb{C}\text{-conjugation; identity}\\
 &  & \mathbb{R}\oplus\mathbb{R}&\text{exchange; identity}\\
\midrule
2 & 4 & \mathbb{H}&\mathbb{H}\text{-conjugation};\ \mathbb{C}\text{-conjugation}\\
 &  & M_2(\mathbb{R})&\text{adjugate; transpose};\ x\mapsto mx^Tm\quad\bigl(m=\left(\begin{smallmatrix}0&1\\1&0\end{smallmatrix}\right)\bigr)\\
 &  & \mathbb{C}\oplus\mathbb{C}&\text{exchange; entrywise }\mathbb{C}\text{-conjugation};\\
 &  & &\text{identity};\ (x,y)\mapsto(\bar x,y)\\
 &  & \mathbb{R}^4&\text{identity};\ (12);\ (12)(34)\\
\bottomrule
\end{array}
\]
\caption{Cayley--Dickson doubling over \(\mathbb{R}\).}
\label{tab:Cayley-Dickson}
\end{table}

\begin{figure}[ht]
\[
\resizebox{\textwidth}{!}{%
\begin{tikzpicture}[
  >=Latex,
  every node/.style={font=\scriptsize, inner sep=2pt},
  classical/.style={->, thin, dashed, draw=gray!90},
  newedge/.style={->, thin},
  lab/.style={font=\tiny, fill=white, inner sep=1pt},
  classlab/.style={font=\tiny, fill=white, inner sep=1pt, text=gray!90},
  classnode/.style={text=gray!90}
]

\node[classnode](R) at (0,0) {\(\mathbb R\)};

\node[classnode] (C)  at (-2.6,-1.8) {\(\mathbb C\)};
\node[classnode] (RR) at ( 2.6,-1.8) {\(\mathbb R\oplus\mathbb R\)};

\node[classnode] (H) at (-5.6,-4.0) {\(\mathbb H\)};
\node[classnode] (M2R) at (-1.9,-4.0) {\(M_2(\mathbb R)\)};
\node (CC)  at ( 1.9,-4.0) {\(\mathbb C\oplus\mathbb C\)};
\node (R4)  at ( 5.6,-4.0) {\(\mathbb R^4\)};

\draw[classical] (R) to[out=-140,in=90]
  node[classlab,pos=.30,left,xshift=-6pt,yshift=1pt] {\(-\)} (C);
\draw[classical] (R) to[out=-40,in=90]
  node[classlab,pos=.38,right,xshift=10pt,yshift=2pt] {\(+\)} (RR);

\draw[classical] (C) to[out=-140,in=90]
  node[classlab,pos=.25,left,xshift=-6pt,yshift=1pt] {\(\alpha,-\)} (H);
\draw[classical] (C) to[out=-100,in=90]
  node[classlab,pos=.72,left,xshift=-6pt,yshift=1pt] {\(\alpha,+\)} (M2R);
\draw[newedge] (C) to[out=-40,in=90]
  node[lab,pos=.30,right,xshift=6pt,yshift=1pt] {\(\beta,\pm\)} (CC);

\draw[classical] (RR) to[out=-140,in=90]
  node[classlab,pos=.25,left,xshift=-6pt,yshift=1pt] {\(\alpha,\pm\)} (M2R);
\draw[newedge] (RR) to[out=-80,in=90]
  node[lab,pos=.72,right,xshift=6pt,yshift=1pt] {\(\beta,-\)} (CC);
\draw[newedge] (RR) to[out=-40,in=90]
  node[lab,pos=.38,right,xshift=10pt,yshift=2pt] {\(\beta,+\)} (R4);
\end{tikzpicture}
}
\]\caption{Cayley--Dickson doubling over \(\mathbb{R}\). Edge labels indicate the involution on the source algebra and the sign of \(\mu\). Dashed gray arrows indicate the classical (\(\alpha\)-only) paths.}
\label{fig:Cayley-Dickson}
\end{figure}
\end{proposition}

\begin{proof}
Let \(K=\mathbb{R}\). From \cite[Section 6]{McC85}, \(\Cay(A,\mu)\) and \(\Cay(A,\lambda^2\mu)\) are isomorphic algebras for any \(\lambda\in\mathbb{R}\setminus\{0\}\). The isomorphism \((x,y)\mapsto(x,\lambda^{-1}y)\) fixes \(A\) pointwise and transports \(A\)-involutions to \(A\)-involutions. Taking \(\lambda=1/\sqrt{|\mu|}\), we may therefore assume that \(\mu=\pm1\).\\

\noindent \(n=0\): The only involution of \(\mathbb{R}\) is the identity involution.\\

\noindent \(n=1\): If \(\mu=-1\), then we recover the classical construction of \(\mathbb{C}\) from \(\mathbb{R}\). The involutions \(\alpha,\beta\colon\mathbb{C}\to\mathbb{C}\) are complex conjugation and the identity involution, respectively. By \autoref{cor:alpha-beta-involutions}, \(\mathbb{C} = \Cay(\mathbb{R},-1)\) has no other involutions. 

If \(\mu=1\), then we obtain the split-complex numbers. An isomorphism of \(\mathbb{R}\)-algebras \(\phi\colon\Cay(\mathbb{R},1)\to\mathbb{R}\oplus\mathbb{R}\) is given by \(\phi(a,b)=(a+b,a-b)\) for all \(a,b\in\mathbb{R}\). The involution \(\alpha\colon\Cay(\mathbb{R},1)\to \Cay(\mathbb{R},1)\) is given by \(\alpha(a,b)=(a,-b)\) for all \(a,b\in\mathbb{R}\). Under the isomorphism \(\phi\), this corresponds to the exchange involution of \(\mathbb{R}\oplus\mathbb{R}\), given by \(\phi\circ\alpha\circ\phi^{-1}(a,b)=(b,a)\) for all \(a,b\in\mathbb{R}\). The involution \(\beta\colon\Cay(\mathbb{R},1)\to \Cay(\mathbb{R},1)\) is the identity involution, and thus corresponds under \(\phi\) to the identity involution of \(\mathbb{R}\oplus\mathbb{R}\). By \autoref{cor:alpha-beta-involutions}, \(\Cay(\mathbb{R},1)\) has no other involutions. \\

\noindent \(n=2\): Let \(a,b\in A_1\) be as in \autoref{thm:inv}. By \ref{lem:1ix} of \autoref{lem:1}, \(a^2=0\). Since neither \(\mathbb{C}\) nor \(\mathbb{R}\oplus\mathbb{R}\) has nonzero nilpotent elements, \(a=0\). By \autoref{prop:inv-2tf}, the remaining conditions are \(b^2=1\) and \(b^*=b\). For \(A_1=\mathbb{C}\), these give \(b=\pm1\). For \(A_1=\mathbb{R}\oplus\mathbb{R}\) with the exchange involution, they give \(b=(1,1)\) or \(b=(-1,-1)\). With the identity involution, all four choices \(b=(\epsilon_1,\epsilon_2)\), where \(\epsilon_1,\epsilon_2\in\{1,-1\}\), are possible.

Starting with \(\mathbb{C}\), if \(\gamma_1=\alpha\colon\mathbb{C}\to\mathbb{C}\) and \(\mu=-1\), then we recover the classical construction of \(\mathbb{H}\cong\{a+bi+cj+dk\mid i^2=j^2=k^2=ijk=-1,\ a,b,c,d\in\mathbb{R}\}\) from \(\mathbb{C}\). The involution \(\alpha\colon\Cay(\mathbb{C},-1)\to\Cay(\mathbb{C},-1)\) is the standard involution of \(\mathbb{H}\), which corresponds to \(i\mapsto -i\), \(j\mapsto -j\), \(k\mapsto-k\). The involution \(\beta\) corresponds to \(i\mapsto-i\), \(j\mapsto j\), \(k=ij\mapsto j(-i)=ij=k\). By \autoref{cor:alpha-beta-involutions}, \(\mathbb{H} = \Cay(\mathbb{C},-1)\) has no other involutions extending complex conjugation. 

Starting with \(\mathbb{C}\), if \(\gamma_1=\alpha\colon\mathbb{C}\to\mathbb{C}\) and \(\mu=1\), then we obtain the split-quaternions. An isomorphism of \(\mathbb{R}\)-algebras \(\phi\colon\Cay(\mathbb{C},1)\to M_2(\mathbb{R})\) is given by \[\phi(a+bi,c+di)=\begin{pmatrix}a+c&d-b\\b+d&a-c\end{pmatrix}\qquad (a,b,c,d\in\mathbb{R}).\]
The involutions \(\alpha,\beta\colon\Cay(\mathbb{C},1)\to\Cay(\mathbb{C},1)\) are given by 
\[
    \alpha(a+bi,c+di)=(a-bi,-c-di) \quad\text{and}\quad
    \beta(a+bi,c+di)=(a-bi,c+di), 
\]
for all \(a,b,c,d\in\mathbb{R}\). Under the isomorphism \(\phi\), \(\alpha\) corresponds to taking the matrix adjugate, while \(\beta\) corresponds to taking the matrix transpose: \(\phi\circ\alpha\circ\phi^{-1}(x)=\operatorname{adj}{x}\) and \(\phi\circ\beta\circ\phi^{-1}(x)=x^T\) for all \(x\in M_2(\mathbb{R})\). By \autoref{cor:alpha-beta-involutions}, \(\Cay(\mathbb{C},1)\) has no other involutions extending complex conjugation. 

Starting with \(\mathbb{C}\), if \(\gamma_1=\beta\colon\mathbb{C}\to\mathbb{C}\), then we obtain an \(\mathbb{R}\)-algebra isomorphic to \(\mathbb{C}\oplus\mathbb{C}\). Indeed, if \(\mu=-1\), then an isomorphism \(\phi\colon \Cay(\mathbb{C},-1)\to\mathbb{C}\oplus\mathbb{C}\) is given by \(\phi(x,y)=(x+yi,x-yi)\) for all \(x,y\in\mathbb{C}\). If \(\mu=1\), then an isomorphism \(\phi\colon \Cay(\mathbb{C},1)\to\mathbb{C}\oplus\mathbb{C}\) is given by \(\phi(x,y)=(x+y,x-y)\) for all \(x,y\in\mathbb{C}\). The involution \(\alpha\colon\Cay(\mathbb{C},1)\to\Cay(\mathbb{C},1)\) is given by \(\alpha(a+bi,c+di)=(a+bi,-c-di)\) for all \(a,b,c,d\in\mathbb{R}\). Under the isomorphism(s) \(\phi\), this corresponds to the exchange involution of \(\mathbb{C}\oplus\mathbb{C}\), given by \(\phi\circ\alpha\circ\phi^{-1}(x,y)=(y,x)\) for all \(x,y\in\mathbb{C}\). The involution \(\beta\colon\Cay(\mathbb{C},1)\to \Cay(\mathbb{C},1)\) is the identity involution, and thus corresponds under \(\phi\) to the identity involution of \(\mathbb{C}\oplus\mathbb{C}\). By \autoref{cor:alpha-beta-involutions}, \(\Cay(\mathbb{C},1)\) has no other involutions extending the identity involution. 

Starting with \(\mathbb{R}\oplus\mathbb{R}\), if \(\gamma_1=\alpha\colon\mathbb{R}\oplus\mathbb{R}\to\mathbb{R}\oplus\mathbb{R}\), then we obtain \(M_2(\mathbb{R})\). An isomorphism \(\phi\colon\Cay(\mathbb{R}\oplus\mathbb{R},1)\to M_2(\mathbb{R})\) is given by
\[\phi\left((a,b),(c,d)\right)=\begin{pmatrix}a&c\\d&b\end{pmatrix}\qquad (a,b,c,d\in\mathbb{R}),\]
while an isomorphism \(\phi\colon\Cay(\mathbb{R}\oplus\mathbb{R},-1)\to M_2(\mathbb{R})\) is given by
\[\phi\left((a,b),(c,d)\right)=\begin{pmatrix}a&c\\-d&b\end{pmatrix}\qquad (a,b,c,d\in\mathbb{R}).\]
The involution \(\alpha\colon\Cay(\mathbb{R}\oplus\mathbb{R},\pm1)\to\Cay(\mathbb{R}\oplus\mathbb{R},\pm1)\) is given by \(\alpha\left((a,b),(c,d)\right)=\left((b,a),(-c,-d)\right)\) for all \(a,b,c,d\in\mathbb{R}\). Under the isomorphism(s) \(\phi\), \(\alpha\) corresponds to taking the matrix adjugate, \(\phi\circ\alpha\circ\phi^{-1}(x)=\operatorname{adj}x\) for all \(x\in M_2(\mathbb{R})\). The involution \(\beta\colon\Cay(\mathbb{R}\oplus\mathbb{R},\pm1)\to\Cay(\mathbb{R}\oplus\mathbb{R},\pm1)\) is given by \(\beta\left((a,b),(c,d)\right)=\left((b,a),(c,d)\right)\) for all \(a,b,c,d\in\mathbb{R}\). Under the isomorphism(s) \(\phi\), applying \(\beta\) corresponds to swapping the diagonal entries, \(\phi\circ\beta\circ\phi^{-1}(x)=mx^Tm\) where
\[m=\begin{pmatrix}0&1\\1&0\end{pmatrix}\in M_2(\mathbb{R}).\] The involution \(\sigma(x)=mx^Tm\) is not \(*\)-isomorphic to matrix transpose: for \(e=\left(\begin{smallmatrix}1&0\\0&0\end{smallmatrix}\right)\ne0\), we have \(e\sigma(e)=0\), whereas \(xx^T=0\) implies \(x=0\). By the above description of \(b\), \(\alpha\) and \(\beta\) are the only involutions of \(\Cay(\mathbb{R}\oplus\mathbb{R},\pm1)\) extending the exchange involution of \(\mathbb{R}\oplus\mathbb{R}\).

For \(\mathbb{R}\oplus\mathbb{R}\) with the identity involution, write
\[
\gamma_{\epsilon_1,\epsilon_2}\left((a,b),(c,d)\right)
=\left((a,b),(\epsilon_1c,\epsilon_2d)\right),
\qquad \epsilon_1,\epsilon_2\in\{1,-1\},
\]
for the four \(\mathbb{R}\oplus\mathbb{R}\)-involutions of its Cayley double.

Starting with \(\mathbb{R}\oplus\mathbb{R}\), if \(\gamma_1=\beta\colon\mathbb{R}\oplus\mathbb{R}\to\mathbb{R}\oplus\mathbb{R}\) and \(\mu=-1\), then we obtain \(\mathbb{C}\oplus\mathbb{C}\). An isomorphism \(\phi\colon\Cay(\mathbb{R}\oplus\mathbb{R},-1)\to \mathbb{C}\oplus\mathbb{C}\) is given by \(\phi\left((a,b),(c,d)\right)=(a+ci,b+di)\) for all \(a,b,c,d\in\mathbb{R}\). The involution \(\alpha\colon\Cay(\mathbb{R}\oplus\mathbb{R},-1)\to\Cay(\mathbb{R}\oplus\mathbb{R},-1)\) is given by \(\alpha\left((a,b),(c,d)\right)=\left((a,b),(-c,-d)\right)\) for all \(a,b,c,d\in\mathbb{R}\). Under the isomorphism \(\phi\), \(\alpha\) corresponds to component-wise complex conjugation, \(\phi\circ\alpha\circ\phi^{-1}(a+bi, c+di)=(a-bi,c-di)\) for all \(a,b,c,d\in\mathbb{R}\). We note that \(\mathbb{C}\oplus\mathbb{C}\) equipped with the latter involution is not \(*\)-isomorphic to \(\mathbb{C}\oplus\mathbb{C}\) with the exchange involution; the latter only has the trivial self-adjoint idempotents \((0,0)\) and \((1,1)\), while the former also has the nontrivial self-adjoint idempotent \((1,0)\). The involution \(\beta\colon\Cay(\mathbb{R}\oplus\mathbb{R},-1)\to\Cay(\mathbb{R}\oplus\mathbb{R},-1)\) is the identity involution and thus corresponds to the identity involution of \(\mathbb{C}\oplus\mathbb{C}\). Under \(\phi\), the involutions \(\gamma_{-1,1}\) and \(\gamma_{1,-1}\) correspond to \((x,y)\mapsto(\bar x,y)\) and \((x,y)\mapsto(x,\bar y)\), respectively. These are \(*\)-isomorphic via the exchange of the two components.

Starting with \(\mathbb{R}\oplus\mathbb{R}\), if \(\gamma_1=\beta\colon\mathbb{R}\oplus\mathbb{R}\to\mathbb{R}\oplus\mathbb{R} \) and \(\mu=1\), then we obtain \(\mathbb{R}^4\). An isomorphism \(\phi\colon\Cay(\mathbb{R}\oplus\mathbb{R},1)\to \mathbb{R}^4\) is given by \(\phi\left((a,b),(c,d)\right)=(a+c,a-c,b+d,b-d)\) for all \(a,b,c,d\in\mathbb{R}\). The involution \(\alpha\colon\Cay(\mathbb{R}\oplus\mathbb{R},1)\to\Cay(\mathbb{R}\oplus\mathbb{R},1)\) is given by \(\alpha\left((a,b),(c,d)\right)=\left((a,b),(-c,-d)\right)\) for all \(a,b,c,d\in\mathbb{R}\). Under the isomorphism \(\phi\), \(\alpha\) corresponds to the block permutation \((12)(34)\) on \(\mathbb{R}^4\), i.e. \(\phi\circ\alpha\circ\phi^{-1}(a,b,c,d)=(b,a,d,c)\) for all \(a,b,c,d\in\mathbb{R}\). The involution \(\beta\colon\Cay(\mathbb{R}\oplus\mathbb{R},1)\to\Cay(\mathbb{R}\oplus\mathbb{R},1)\) is the identity involution and thus corresponds to the identity involution of \(\mathbb{R}^4\). Under \(\phi\), the involutions \(\gamma_{-1,1}\) and \(\gamma_{1,-1}\) correspond to \((12)\) and \((34)\), respectively. These are \(*\)-isomorphic via the permutation \((13)(24)\).

In the order listed in \autoref{tab:Cayley-Dickson}, the fixed subspaces have real dimensions \(1,3\) on \(\mathbb{H}\); \(1,3,3\) on \(M_2(\mathbb{R})\); \(2,2,4,3\) on \(\mathbb{C}\oplus\mathbb{C}\); and \(4,3,2\) on \(\mathbb{R}^4\). The pairs with equal dimensions on the same algebra were distinguished above. Thus there are exactly twelve \(*\)-algebra isomorphism classes in dimension \(4\).
\end{proof}

\begin{remark}
  A natural problem is to characterize the \(*\)-algebras \(A_n\), constructed in \autoref{prop:cay-over-R}, for any \(n\geq 3\). If \(A_2=\mathbb{H}\) or \(A_2=M_2(\mathbb{R})\), then \(\alpha\) and \(\beta\) are the only \(A_n\)-involutions of \(\Cay(A_n,\mu_n)\) for \(n\geq3\). Indeed, \(C(A_2)=Z(A_2)=\mathbb{R}\), and each \(A_n\) contains the noncommutative algebra \(A_2\), so \(\Sk(A_n)\ne0\). Thus \autoref{prop:ncz} gives \(N(A_n)=C(A_n)=\mathbb{R}\) for \(n\geq3\) by induction. Together with \(S(A_n)=0\) from \autoref{prop:extension-I(A)}, \autoref{prop:inv-2tf} yields \(a=0\) and \(b=\pm1\). The restriction \(n\geq3\) is necessary in general: for \(A_2=M_2(\mathbb{R})\) with transpose, \(\gamma(x,y)=(x^T,by)\), where \(b=\left(\begin{smallmatrix}1&0\\0&-1\end{smallmatrix}\right)\), is an additional \(A_2\)-involution by \autoref{prop:inv-2tf}.
\end{remark}

\end{document}